\documentclass[a4paper,12pt]{article}
\usepackage{a4}
\usepackage{amsmath}
\usepackage{amsfonts}
\usepackage{amsthm}
\usepackage{amssymb}
\usepackage{authblk}
\usepackage[usenames,dvipsnames,svgnames,table]{xcolor}
\usepackage{graphicx}
\usepackage{pst-eucl}
\usepackage{calc}
\usepackage{pst-3dplot}%
\usepackage{pst-grad}
\usepackage{pst-plot,pst-math,pstricks-add}%
\usepackage{pst-all}
\usepackage{multirow}
\makeatletter
\newtheorem{theorem}{Theorem}[section]
\newtheorem{remark}[theorem]{Remark}

\newtheorem{corollary}[theorem]{Corollary}
\newtheorem{proposition}[theorem]{Proposition}
\newtheorem{example}[theorem]{Example}

\title{Spectra and Laplacian spectra of arbitrary powers of lexicographic products of graphs}

\author[1]{Nair  Abreu}
\author[2,3]{Domingos M. Cardoso}
\author[2,3]{Paula Carvalho}
\author[4]{Cybele T. M. Vinagre}

\affil[1]{\small
PEP/COPPE, Universidade Federal do Rio de Janeiro, Rio de Janeiro, Brasil. Email: nairabreunovoa@gmail.com}
\affil[2]{\small Centro de Investiga\c{c}\~{a}o e Desenvolvimento em Matem\'atica e Aplica\c{c}\~{o}es}
\affil[3]{\small
Departamento de Matem\'atica, Universidade de  Aveiro, 3810-193, Aveiro, Portugal. Email: (dcardoso,paula.carvalho)@ua.pt}
\affil[4]{\small
Instituto de Matem\'atica e Estat\'{\i}stica, Universidade Federal Fluminense, Niter\'{o}i, Brasil. Email: cybl@vm.uff.br}

\begin{document}

\maketitle

\begin{abstract}
Consider two graphs $G$  and $H$. Let $H^k[G]$ be  the lexicographic product  of $H^k$ and $G$, where
$H^k$ is the  lexicographic product of the graph $H$ by itself $k$ times. In this paper, we determine the
spectrum of $H^k[G]$ and $H^k$ when $G$ and $H$ are regular and the Laplacian spectrum of $H^k[G]$ and
$H^k$ for $G$ and $H$ arbitrary. Particular emphasis is given to the least eigenvalue of the
adjacency matrix in the case of lexicographic powers of regular graphs, and to the algebraic connectivity
and the largest Laplacian eigenvalues in the case of lexicographic powers of arbitrary graphs. This
approach allows the determination of the spectrum (in case of regular graphs) and Laplacian spectrum
(for arbitrary graphs) of huge graphs. As an example, the spectrum of the lexicographic power of the Petersen
graph with the googol number (that is, $10^{100}$) of vertices is determined. The paper finish with the
extension of some well known spectral and combinatorial invariant properties of graphs to its lexicographic
powers.
\end{abstract}

\section{Introduction}
The lexicographic product of a graph $H$ by himself several times is a very special graph product, it is a kind of
fractal graph which reproduces its copy in each of the positions of its vertices and connects all  the vertices of each
copy with another copy when they are placed in positions corresponding to adjacent vertices of $H$. This procedure can be repeated,
reproducing a copy of the previous iterated graph in each of the positions of the vertices of $H$ and so on.
Despite the spectrum and Laplacian spectrum of the lexicographic product of two graphs (with some restrictions regarding the spectrum)
expressed in terms of the two factors be well known (see \cite{CAER2013}, where a unified approach is given), it is not the case
of the spectra and Laplacian spectra of graphs obtained by iterated lexicographic products, herein called lexicographic powers,
of regular and arbitrary graphs, respectively. A lexicographic power $H^k$ of a graph $H$ can produce a graph with a huge number of vertices
whose spectra and  Laplacian spectra may not be determined using their adjacency and Laplacian matrices, respectively.  The expressions herein
deduced for the spectra and Laplacian spectra of lexicographic powers can be easily programmed, for example, in Mathematica, and the results can
be obtained immediately. For instance, the spectrum of the $100$-th lexicographic power of the Petersen graph, presented in Section~\ref{main},
was obtained by Mathematica and the computations lasted only a few seconds. Notice that such lexicographic power has the googol number (that is,
$10^{100}$) of vertices.\\

The paper is organized as follows. In the next section, the notation is introduced and some preliminary results are given. The main results are
introduced in Section~\ref{main}, where the spectra (Laplacian spectra) of $H^k[G]$ and $H^k$, when $G$ and $H$ are regular (arbitary) graphs,
are deduced. Particular attention is given to the Laplacian index and algebraic connectivity of the lexicographic powers of arbitrary graphs. In
Section~\ref{extensions}, the obtained results are applied to extend some well known properties and spectral relations of combinatorial invariants
of graphs $H$ to its lexicographic powers $H^k$.

\section{Preliminaries}\label{preliminaries}
In this work we deal with simple and undirected graphs. If $G$ is such a graph of order $n$, its vertex set is denoted by $V(G)$ and its edge set by $E(G)$. The elements of $E(G)$ are denoted by $ij$,
where $i$ and $j$ are the extreme vertices of the edge $ij$. The degree of $j \in V\left( G\right)$ is denoted by $d_G\left( j\right)$, the maximum and minimum degree of the vertices in $G$ are
$\delta(G)$ and $\Delta(G)$ and the set of the neighbors of a vertex $j$ is $N_{G}(j)$.
The \emph{adjacency matrix} of $G$ is the $n \times n$ matrix $A_G$ whose $(i,j)$-entry is equal to $1$ whether $ij \in E(G)$
and is equal to $0$ otherwise. The \emph{Laplacian  matrix} of $G$ is the matrix $L_G=D-A_G$, where $D$ is the diagonal matrix whose diagonal elements
are the degrees of the vertices of $G$.
Since  $A_G$ and $L_G$ are symmetric matrices, theirs eigenvalues are real numbers.
From Ger\v{s}gorin's theorem,  the eigenvalues of $L_G $
are nonnegative.
 The multiset (that is, the set with possible repetitions) of eigenvalues of a matrix $M$ is called the
\emph{spectrum } of $M$ and denoted $\sigma(M)$.
 Throughout the paper, we write  $\sigma _A(G) = \{\lambda_1^{[g_1]}, \ldots, \lambda_s^{[g_s]}\}$ (respectively, $\sigma _{L}(G)  = \{\mu_1^{[l_1]}, \ldots, \mu_t^{[l_t]}\}$) when $\lambda_1> \ldots >\lambda_s$  ( $\mu_1 > \ldots > \mu_t$ ) are the distinct eigenvalues of $A_G$ ($L_G$) indexed in decreasing order - in this case,
 $\gamma^{[r]}$  means that the eigenvalue  $\gamma$  has multiplicity $r$. If  convenient, we write $\gamma(G)$
 in place of $\gamma$  to indicate  an eigenvalue of a matrix associated to $G$, and
  we denote the eigenvalues of $A_G $ (respectively, $L_G $) indexed  in non increasing order, as  $\lambda_1(G) \geq \cdots \geq \lambda_n(G)$
($\mu_1(G) \geq \cdots \geq \mu_n(G)$).

As usually, the adjacency matrix eigenvalues of a graph $G$ are called the \emph{eigenvalues of} $G$.
We remember  that $\mu_n(G)=0$ (the all one vector is the associated eigenvector) and its multiplicity is equal to the
number of components of $G$. Besides,  $\mu_{n-1}(G)$ is called the \emph{algebraic connectivity} of $G$ \cite{Fiedler73}.
Further concepts not defined in this paper can be found in \cite{Cve1, Cve2}.

The \textit{lexicographic product} (also called the \textit{composition}) of the graphs $H$ and $G$ is the graph
$H[G]$ (also denoted by $H \circ G$) for which the  vertex set is the cartesian product $V(H) \times V(G)$ and such that  a
vertex $(x_1,y_1)$ is adjacent to the vertex $(x_2,y_2)$ whenever $x_1$ is adjacent to $x_2$ or $x_1=x_2$ and $y_1$
is adjacent to $y_2$ (see \cite{harary_94} and \cite{HIK} for notations and further details). This graph operation was introduced by Harary in \cite{harary_59} and
Sabidussi in \cite{sabidussi_59}. It is immediate that the lexicographic product is associative but it is not
commutative.

The lexicographic product was generalized in \cite{Schwenk74} as follows: consider a graph $H$ of order $n$
and graphs  $G_i$, $i=1,\ldots,n$, with vertex sets  $V(G_i)$s   two by two disjoints.
The \textit{generalized composition} $H[G_1,  \ldots, G_n]$ is the graph such that
\begin{eqnarray*}
V(H[G_1,  \ldots, G_n]) &=& \bigcup_{i=1}^{n}{V(G_i)}  \   \  \   \mbox{ and } \\
E(H[G_1,  \ldots, G_n]) &=& \bigcup_{i=1}^{n}{E(G_i)}\ \cup \bigcup_{ij  \in E(H)}{E(G_i \vee G_j)},
\end{eqnarray*}
where $G_i \vee G_j$ denotes the join of the graphs $G_i$ and $G_j$. This operation is called in \cite{CAER2013}
the $H$\emph{-join} of graphs $G_1,  \ldots, G_n$.
In \cite{Schwenk74} and \cite{CAER2013}, the spectrum of $H[G_1,  \ldots, G_n]$ is provided, where $H$
is an arbitrary graph and  $G_1,  \dots, G_n$ are regular graphs. Furthermore, in \cite{G2010} and
\cite{CAER2013}, using different approaches, it was characterized the spectrum of the Laplacian matrix of $H[G_1,  \dots, G_n]$
for arbitrary graphs.

Now, let us focus on the spectrum of the adjacency and Laplacian matrix of the above generalized graph composition.

\begin{theorem}\cite{CAER2013}\label{CAER_th_1}
Let $H$ be a graph such that $V(H)=\{1, \dots, n\}$ and, for  $j = 1, \ldots, n$, let $G_j$ be a $p_j$-regular
graph with order $m_j$. Then
\begin{equation}
\sigma_A(H[G_1, \dots, G_n]) = \left(\bigcup_{j=1}^{n}{\left(\sigma_A(G_j) \setminus \{p_j\}\right)}\right)
                                    \cup \sigma({C}), \label{marca_dmc1}
\end{equation}
where
\begin{equation}
{C} = \begin{pmatrix}
                       p_1       & c_{12}    & \ldots & c_{1(n-1)}& c_{1n} \\
                   c_{21}     & p_2          & \ldots & c_{2(n-1)}& c_{2n} \\
                      \vdots     & \vdots       & \ddots & \vdots       & \vdots \\
                   c_{(n-1)1} & c_{(n-1)2}& \ldots & p_{n-1}      & c_{(n-1)n} \\
                   c_{n1}     & c_{n2}    & \ldots & c_{n(n-1)}& p_n
\end{pmatrix} \label{marca_dmc2}
\end{equation}
and
\begin{equation}
c_{ij} = \left\{\begin{array}{lll}
                   \sqrt{m_i m_j}  &  & \hbox{if } ij \in E(H), \\
                   0               &  & \hbox{otherwise.}%
                  \end{array}\right.  \label{marca_dmc3}
\end{equation}
\end{theorem}

Let $H$ be a graph of order $n$ and $G$  be an arbitrary graph. If, for $1 \le i \le n$,  $G_i$ is isomorphic to $G$,
it follows immediately that $H[G_1, \ldots, G_n]=H[G]$, a fact also noted in \cite{BaBaPa}. In particular case of a
regular graph $G$, Theorem \ref{CAER_th_1} implies the corollary below.

\begin{corollary}\label{corollary_1}
Let $H$ be  a graph of order $n$ with $\sigma _A(H) = \{\lambda^{[h_1]}_1(H), \dots, \lambda^{[h_t]}_t(H)\}$ and
let $G$ be a  $p$-regular graph of order $m$ such that $\sigma _A(G) = \{\lambda^{[g_1]}_1(G), \ldots, \lambda^{[g_s]}_s(G)\}$. Then
$$
\sigma_A(H[G])= \{p^{[n(g_1-1)]}, \ldots, \lambda^{[ng_s]}_s(G)\} \cup \{(m{\lambda_1(H)}+p)^{[h_1]}, \ldots, (m{\lambda_t(H)}+p)^{[h_t]}\}\,.
$$
\end{corollary}
\vskip 0.3cm

Now it is worth to recall the following result.

\begin{theorem}\cite{CAER2013}\label{CAER_th_2}
Let $H$ be a graph such that $V(H)=\{1, \ldots, n\}$ and, for each $j \in \{1, \ldots, n\}$, let  $G_{j}$  be a graph of
order $m_{j}$ with Laplacian spectrum $\sigma_L(G_{j})$. Then the Laplacian spectrum of $H[G_1, \dots, G_n]$ is given by
\begin{equation*}
\sigma_L(H[G_1, \dots, G_n]) = \left( \bigcup_{j=1}^{n}{(s_j + (\sigma_L(G_{j}) \setminus \{0\}))}\right) \cup \sigma({C}),
\end{equation*}
where
\begin{equation*}
s_j = \left\{\begin{array}{ll}
      \sum_{i \in N_H(j)}{m_i}, & \hbox{if } N_H(j) \ne \emptyset, \\
                             0, & \hbox{otherwise}%
\end{array}\right.\,
\end{equation*}
and
$s_j + (\sigma_L(G_{j}) \setminus \{0\})$
means that the number $s_j$ is added to each element of $\sigma_L(G_{j}) \setminus \{0\}$, and
\begin{equation}
{C} = \begin{pmatrix}
                       s_1       &-c_{12}    & \ldots &-c_{1(n-1)}&-c_{1n} \\
                   -c_{21}    & s_2          & \ldots &-c_{2(n-1)}&-c_{2n} \\
                      \vdots     & \vdots       & \ddots & \vdots       & \vdots \\
                   -c_{(n-1)1}&-c_{(n-1)2}& \ldots & s_{n-1}      &-c_{(n-1)n} \\
                   -c_{n1}    &-c_{n2}    & \ldots &-c_{n(n-1)}& s_n
\end{pmatrix}\label{marca_dmc2}   \  \  \  \  \mbox{with}
\end{equation}
\begin{equation}
c_{ij} = \left\{\begin{array}{lll}
                   \sqrt{m_i m_j}  &  & \hbox{if } ij \in E(H), \\
                   0               &  & \hbox{otherwise}.
                  \end{array}\right.  \label{marca_dmc3}
\end{equation}
 \end{theorem}

Assuming that $G_1,  \ldots, G_n$ are all isomorphic to a particular graph $G$ we have the next corollary, which was also proved in \cite{BaBaPa}.

\begin{corollary}\label{corollary_2}
Let $H$ be  a graph of order $n$ with $\sigma _L(H) = \{{\mu_1}(H), \ldots, {\mu_n}(H)\}$ and
let $G$ be a
graph of order $m$ such that $\sigma _L(G) = \{{\mu_1(G), \ldots,  {\mu_m}(G)}\}$. Then
$$
\sigma_L(H[G])=\left( \bigcup_{j=1}^{n}{\left\{md_H(j)+\mu_i(G) : 1 \leq i \leq m-1\right\}}\right)\cup \left\{m\mu_1(H), \ldots,  m\mu_n(H)\right\}.$$
\end{corollary}

\section{The spectra and Laplacian spectra of i\-te\-rated lexicographic products of graphs}\label{main}

Let us consider the graphs obtained by an arbitrary number of iterations of the lexicographic product  of a graph
by another as follows:
$$
H^0[G]=G,  \ H^1[G]=H[G] \  \mbox{ and }  \  H^k[G]=H[H^{k-1}[G]],  \mbox{ for all integer } \  k \geq 2.
$$

\begin{example}\label{ex_1}
Let us consider the graph $H=C_4$ (the cycle with four vertices) and $G=K_2$ (the complete graph
with two vertices). Then $H^{0}[G] = K_2$ and $H[G] = C_4[K_2]$ are depicted in Figure~\ref{graphG1}.
Furthermore, the Figure~\ref{graphG2}
depicts the graph $H^2[G] = {C_4}^2[K_2]= C_{4}[C_{4}[K_{2}]]$.
\end{example}

In what follows, we adopt  the traditional notation of the union of sets for denoting the union of multisets, where the repeated elements of the multisets $A$ and $B$ appear in
$A \cup B$ as many times as we count them in $A$ and $B$.

\begin{center}
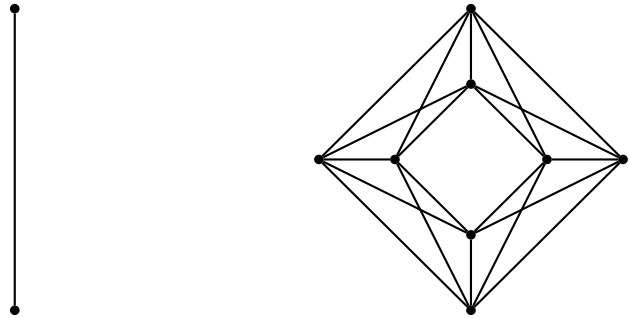
\begin{figure}[h]
\begin{pspicture}(10,4)(-2,0)
\dotnode(2,4){A}
\dotnode(2,0){B}
\ncline{A}{B}
\dotnode(7,2){A11}
\dotnode(6,2){B11}
\dotnode(8,1){A12}
\dotnode(8,0){B12}
\dotnode(9,2){A13}
\dotnode(10,2){B13}
\dotnode(8,3){A14}
\dotnode(8,4){B14}
\ncline{A11}{B11}
\ncline{A12}{B12}
\ncline{A13}{B13}
\ncline{A14}{B14}
\ncline{A11}{A12}
\ncline{A11}{A14}
\ncline{A13}{A12}
\ncline{A13}{A14}
\ncline{B11}{B12}
\ncline{B11}{B14}
\ncline{B13}{B12}
\ncline{B13}{B14}
\ncline{A11}{B14}
\ncline{A11}{B12}
\ncline{B11}{A14}
\ncline{B11}{A12}
\ncline{A13}{B14}
\ncline{A13}{B12}
\ncline{B13}{A14}
\ncline{B13}{A12}
\end{pspicture}
\caption{The graphs $H^0[G]=K_2$ and $H^1[G]=C_4[K_2]$.} \label{graphG1}
\end{figure}
 \end{center}

\vspace{2.5cm}

\begin{figure}[h]
\begin{pspicture}(15,8)(-2,0)
\dotnode(0.5,4){A11}
\dotnode(0,4){B11}
\dotnode(1,3.5){A12}
\dotnode(1,3){B12}
\dotnode(1.5,4){A13}
\dotnode(2,4){B13}
\dotnode(1,4.5){A14}
\dotnode(1,5){B14}
\ncline[linecolor=red]{A11}{B11}
\ncline[linecolor=red]{A12}{B12}
\ncline[linecolor=red]{A13}{B13}
\ncline[linecolor=red]{A14}{B14}
\ncline[linecolor=red]{A11}{A12}
\ncline[linecolor=red]{A11}{A14}
\ncline[linecolor=red]{A13}{A12}
\ncline[linecolor=red]{A13}{A14}
\ncline[linecolor=red]{B11}{B12}
\ncline[linecolor=red]{B11}{B14}
\ncline[linecolor=red]{B13}{B12}
\ncline[linecolor=red]{B13}{B14}
\ncline[linecolor=red]{A11}{B14}
\ncline[linecolor=red]{A11}{B12}
\ncline[linecolor=red]{B11}{A14}
\ncline[linecolor=red]{B11}{A12}
\ncline[linecolor=red]{A13}{B14}
\ncline[linecolor=red]{A13}{B12}
\ncline[linecolor=red]{B13}{A14}
\ncline[linecolor=red]{B13}{A12}
\dotnode(5,7){A41}
\dotnode(4.5,7){B41}
\dotnode(5.5,6.5){A42}
\dotnode(5.5,6){B42}
\dotnode(6,7){A43}
\dotnode(6.5,7){B43}
\dotnode(5.5,7.5){A44}
\dotnode(5.5,8){B44}
\ncline[linecolor=red]{A41}{B41}
\ncline[linecolor=red]{A42}{B42}
\ncline[linecolor=red]{A43}{B43}
\ncline[linecolor=red]{A44}{B44}
\ncline[linecolor=red]{A41}{A42}
\ncline[linecolor=red]{A41}{A44}
\ncline[linecolor=red]{A43}{A42}
\ncline[linecolor=red]{A43}{A44}
\ncline[linecolor=red]{B41}{B42}
\ncline[linecolor=red]{B41}{B44}
\ncline[linecolor=red]{B43}{B42}
\ncline[linecolor=red]{B43}{B44}
\ncline[linecolor=red]{A41}{B44}
\ncline[linecolor=red]{A41}{B42}
\ncline[linecolor=red]{B41}{A44}
\ncline[linecolor=red]{B41}{A42}
\ncline[linecolor=red]{A43}{B44}
\ncline[linecolor=red]{A43}{B42}
\ncline[linecolor=red]{B43}{A44}
\ncline[linecolor=red]{B43}{A42}

\dotnode(5,1){A21}
\dotnode(4.5,1){B21}
\dotnode(5.5,0.5){A22}
\dotnode(5.5,0){B22}
\dotnode(6,1){A23}
\dotnode(6.5,1){B23}
\dotnode(5.5,1.5){A24}
\dotnode(5.5,2){B24}
\ncline[linecolor=red]{A21}{B21}
\ncline[linecolor=red]{A22}{B22}
\ncline[linecolor=red]{A23}{B23}
\ncline[linecolor=red]{A24}{B24}
\ncline[linecolor=red]{A21}{A22}
\ncline[linecolor=red]{A21}{A24}
\ncline[linecolor=red]{A23}{A22}
\ncline[linecolor=red]{A23}{A24}
\ncline[linecolor=red]{B21}{B22}
\ncline[linecolor=red]{B21}{B24}
\ncline[linecolor=red]{B23}{B22}
\ncline[linecolor=red]{B23}{B24}
\ncline[linecolor=red]{A21}{B24}
\ncline[linecolor=red]{A21}{B22}
\ncline[linecolor=red]{B21}{A24}
\ncline[linecolor=red]{B21}{A22}
\ncline[linecolor=red]{A23}{B24}
\ncline[linecolor=red]{A23}{B22}
\ncline[linecolor=red]{B23}{A24}
\ncline[linecolor=red]{B23}{A22}

\dotnode(9.5,4){A31}
\dotnode(9,4){B31}
\dotnode(10,3.5){A32}
\dotnode(10,3){B32}
\dotnode(10.5,4){A33}
\dotnode(11,4){B33}
\dotnode(10,4.5){A34}
\dotnode(10,5){B34}
\ncline[linecolor=red]{A31}{B31}
\ncline[linecolor=red]{A32}{B32}
\ncline[linecolor=red]{A33}{B33}
\ncline[linecolor=red]{A34}{B34}
\ncline[linecolor=red]{A31}{A32}
\ncline[linecolor=red]{A31}{A34}
\ncline[linecolor=red]{A33}{A32}
\ncline[linecolor=red]{A33}{A34}
\ncline[linecolor=red]{B31}{B32}
\ncline[linecolor=red]{B31}{B34}
\ncline[linecolor=red]{B33}{B32}
\ncline[linecolor=red]{B33}{B34}
\ncline[linecolor=red]{A31}{B34}
\ncline[linecolor=red]{A31}{B32}
\ncline[linecolor=red]{B31}{A34}
\ncline[linecolor=red]{B31}{A32}
\ncline[linecolor=red]{A33}{B34}
\ncline[linecolor=red]{A33}{B32}
\ncline[linecolor=red]{B33}{A34}
\ncline[linecolor=red]{B33}{A32}
\ncline{A11}{A21}
\ncline{A11}{A22}
\ncline{A11}{A23}
\ncline{A11}{A24}
\ncline{A11}{B21}
\ncline{A11}{B22}
\ncline{A11}{B23}
\ncline{A11}{B24}
\ncline{A12}{A21}
\ncline{A12}{A22}
\ncline{A12}{A23}
\ncline{A12}{A24}
\ncline{A12}{B21}
\ncline{A12}{B22}
\ncline{A12}{B23}
\ncline{A12}{B24}
\ncline{A13}{A21}
\ncline{A13}{A22}
\ncline{A13}{A23}
\ncline{A13}{A24}
\ncline{A13}{B21}
\ncline{A13}{B22}
\ncline{A13}{B23}
\ncline{A13}{B24}
\ncline{A14}{A21}
\ncline{A14}{A22}
\ncline{A14}{A23}
\ncline{A14}{A24}
\ncline{A14}{B21}
\ncline{A14}{B22}
\ncline{A14}{B23}
\ncline{A14}{B24}%
\ncline{B11}{A21}
\ncline{B11}{A22}
\ncline{B11}{A23}
\ncline{B11}{A24}
\ncline{B11}{B21}
\ncline{B11}{B22}
\ncline{B11}{B23}
\ncline{B11}{B24}
\ncline{B12}{A21}
\ncline{B12}{A22}
\ncline{B12}{A23}
\ncline{B12}{A24}
\ncline{B12}{B21}
\ncline{B12}{B22}
\ncline{B12}{B23}
\ncline{B12}{B24}
\ncline{B13}{A21}
\ncline{B13}{A22}
\ncline{B13}{A23}
\ncline{B13}{A24}
\ncline{B13}{B21}
\ncline{B13}{B22}
\ncline{B13}{B23}
\ncline{B13}{B24}
\ncline{B14}{A21}
\ncline{B14}{A22}
\ncline{B14}{A23}
\ncline{B14}{A24}
\ncline{B14}{B21}
\ncline{B14}{B22}
\ncline{B14}{B23}
\ncline{B14}{B24}
\ncline{A11}{A41}
\ncline{A11}{A42}
\ncline{A11}{A43}
\ncline{A11}{A44}
\ncline{A11}{B41}
\ncline{A11}{B42}
\ncline{A11}{B43}
\ncline{A11}{B44}
\ncline{A12}{A41}
\ncline{A12}{A42}
\ncline{A12}{A43}
\ncline{A12}{A44}
\ncline{A12}{B41}
\ncline{A12}{B42}
\ncline{A12}{B43}
\ncline{A12}{B44}
\ncline{A13}{A41}
\ncline{A13}{A42}
\ncline{A13}{A43}
\ncline{A13}{A44}
\ncline{A13}{B41}
\ncline{A13}{B42}
\ncline{A13}{B43}
\ncline{A13}{B44}
\ncline{A14}{A41}
\ncline{A14}{A42}
\ncline{A14}{A43}
\ncline{A14}{A44}
\ncline{A14}{B41}
\ncline{A14}{B42}
\ncline{A14}{B43}
\ncline{A14}{B44}%
\ncline{B11}{A41}
\ncline{B11}{A42}
\ncline{B11}{A43}
\ncline{B11}{A44}
\ncline{B11}{B41}
\ncline{B11}{B42}
\ncline{B11}{B43}
\ncline{B11}{B44}
\ncline{B12}{A41}
\ncline{B12}{A42}
\ncline{B12}{A43}
\ncline{B12}{A44}
\ncline{B12}{B41}
\ncline{B12}{B42}
\ncline{B12}{B43}
\ncline{B12}{B44}
\ncline{B13}{A41}
\ncline{B13}{A42}
\ncline{B13}{A43}
\ncline{B13}{A44}
\ncline{B13}{B41}
\ncline{B13}{B42}
\ncline{B13}{B43}
\ncline{B13}{B44}
\ncline{B14}{A41}
\ncline{B14}{A42}
\ncline{B14}{A43}
\ncline{B14}{A44}
\ncline{B14}{B41}
\ncline{B14}{B42}
\ncline{B14}{B43}
\ncline{B14}{B44}
\ncline{A31}{A21}
\ncline{A31}{A22}
\ncline{A31}{A23}
\ncline{A31}{A24}
\ncline{A31}{B21}
\ncline{A31}{B22}
\ncline{A31}{B23}
\ncline{A31}{B24}
\ncline{A32}{A21}
\ncline{A32}{A22}
\ncline{A32}{A23}
\ncline{A32}{A24}
\ncline{A32}{B21}
\ncline{A32}{B22}
\ncline{A32}{B23}
\ncline{A32}{B24}
\ncline{A33}{A21}
\ncline{A33}{A22}
\ncline{A33}{A23}
\ncline{A33}{A24}
\ncline{A33}{B21}
\ncline{A33}{B22}
\ncline{A33}{B23}
\ncline{A33}{B24}
\ncline{A34}{A21}
\ncline{A34}{A22}
\ncline{A34}{A23}
\ncline{A34}{A24}
\ncline{A34}{B21}
\ncline{A34}{B22}
\ncline{A34}{B23}
\ncline{A34}{B24}%
\ncline{B31}{A21}
\ncline{B31}{A22}
\ncline{B31}{A23}
\ncline{B31}{A24}
\ncline{B31}{B21}
\ncline{B31}{B22}
\ncline{B31}{B23}
\ncline{B31}{B24}
\ncline{B32}{A21}
\ncline{B32}{A22}
\ncline{B32}{A23}
\ncline{B32}{A24}
\ncline{B32}{B21}
\ncline{B32}{B22}
\ncline{B32}{B23}
\ncline{B32}{B24}
\ncline{B33}{A21}
\ncline{B33}{A22}
\ncline{B33}{A23}
\ncline{B33}{A24}
\ncline{B33}{B21}
\ncline{B33}{B22}
\ncline{B33}{B23}
\ncline{B33}{B24}
\ncline{B34}{A21}
\ncline{B34}{A22}
\ncline{B34}{A23}
\ncline{B34}{A24}
\ncline{B34}{B21}
\ncline{B34}{B22}
\ncline{B34}{B23}
\ncline{B34}{B24}
\ncline{A31}{A41}
\ncline{A31}{A42}
\ncline{A31}{A43}
\ncline{A31}{A44}
\ncline{A31}{B41}
\ncline{A31}{B42}
\ncline{A31}{B43}
\ncline{A31}{B44}
\ncline{A32}{A41}
\ncline{A32}{A42}
\ncline{A32}{A43}
\ncline{A32}{A44}
\ncline{A32}{B41}
\ncline{A32}{B42}
\ncline{A32}{B43}
\ncline{A32}{B44}
\ncline{A33}{A41}
\ncline{A33}{A42}
\ncline{A33}{A43}
\ncline{A33}{A44}
\ncline{A33}{B41}
\ncline{A33}{B42}
\ncline{A33}{B43}
\ncline{A33}{B44}
\ncline{A34}{A41}
\ncline{A34}{A42}
\ncline{A34}{A43}
\ncline{A34}{A44}
\ncline{A34}{B41}
\ncline{A34}{B42}
\ncline{A34}{B43}
\ncline{A34}{B44}%
\ncline{B31}{A41}
\ncline{B31}{A42}
\ncline{B31}{A43}
\ncline{B31}{A44}
\ncline{B31}{B41}
\ncline{B31}{B42}
\ncline{B31}{B43}
\ncline{B31}{B44}
\ncline{B32}{A41}
\ncline{B32}{A42}
\ncline{B32}{A43}
\ncline{B32}{A44}
\ncline{B32}{B41}
\ncline{B32}{B42}
\ncline{B32}{B43}
\ncline{B32}{B44}
\ncline{B33}{A41}
\ncline{B33}{A42}
\ncline{B33}{A43}
\ncline{B33}{A44}
\ncline{B33}{B41}
\ncline{B33}{B42}
\ncline{B33}{B43}
\ncline{B33}{B44}
\ncline{B34}{A41}
\ncline{B34}{A42}
\ncline{B34}{A43}
\ncline{B34}{A44}
\ncline{B34}{B41}
\ncline{B34}{B42}
\ncline{B34}{B43}
\ncline{B34}{B44}

\end{pspicture}
\caption{The graph $H^2[G] = C_4^2[K_2]$.} \label{graphG2}
\end{figure}
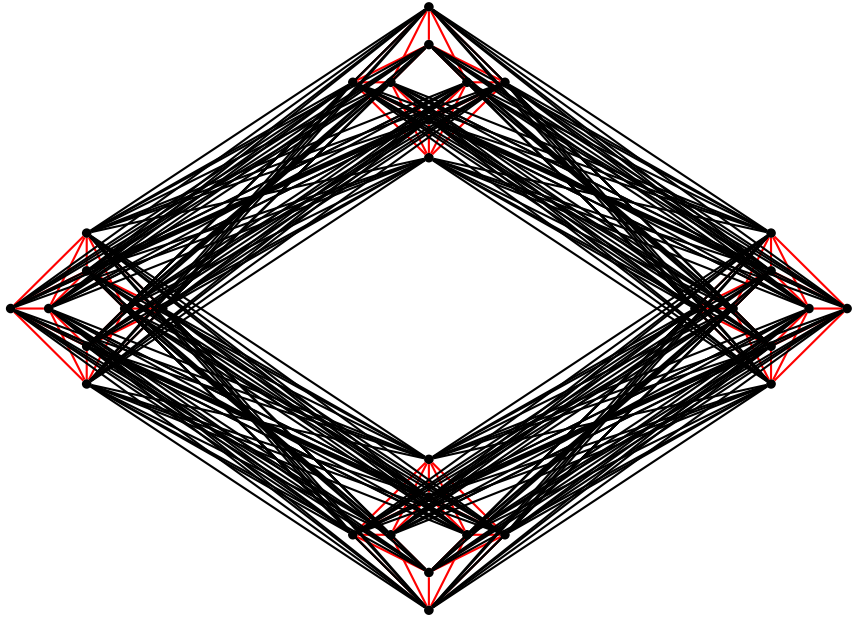

\subsection{The spectrum in the case of a $p$-regular graph $G$ and a $q$-regular graph $H$}

The next theorem states the regularity degree, order and spectrum of $H^k[G]$, for $k \ge 0$,
when $G$ and $H$ are both regular connected graphs.

\begin{theorem}\label{main_th_1}
Let  $H$ be a $q$-regular connected graph of order $n$ with 
$\sigma_A(H)=\{q,\gamma^{[h_2]}_2(H),\dots,\gamma^{[h_t]}_t(H)\}$
and $G$ be a $p$-regular connected graph of order $m$ with 
$\sigma_A(G)=\{p, \gamma^{[g_2]}_2(G), \dots, \gamma^{[g_s]}_s(G)\}$.
Then for each integer $k \ge 0$,  $H^k[G]$ is a $r_k$-regular graph of order $\nu_k$ with
\begin{eqnarray*}
 r_k                         &=& m q \frac{n^k-1}{n-1} + p,\\
\nu_k                       &=& m n^k,\\
\sigma_A(H^k[G])      &=& \left\{\gamma^{[n^kg_2]}_2(G),\dots, \gamma^{[n^kg_s]}_s(G)\right\} \cup
\{r_k\} \cup \Lambda_k \;\; \mbox{  where } \\
\Lambda_k  &=& \bigcup_{i=0}^{k-1}{\left\{(m n^{i}\gamma_2(H)+r_i)^{[n^{k-1-i}h_2]},\dots, (m n^i\gamma_t(H)+r_i)^{[n^{k-1-i}h_t]}\right\}},
\end{eqnarray*}
assuming that $\Lambda_0= \emptyset$.
\end{theorem}
\begin{proof}
Since  $H^0[G]=G$, the case $k=0$ follows. Furthermore, the case $k=1$ follows from Corollary~\ref{corollary_1}, since
$H^1[G]=H[G]$ (notice that $r_1 = mq+p =m\gamma_1(H)+p$).
Let us assume that the result holds for  $k-1$ iterations, with $k \geq 2$. By definition of lexicographic product, we obtain
\begin{eqnarray*}
r_k &=& \nu_{k-1} q + r_{k-1}\\
    &=& m q(n^{k-1}+n^{k-2}+\dots+n+1)+p = mq\frac{n^k-1}{n-1}+p
\end{eqnarray*}
and $\nu_k=\nu(H^k[G])=\nu_{k-1} n = m n^k.$ Additionally, replacing in the Corollary~\ref{corollary_1}
the graph $G$ by $H^{k-1}[G]$ 
 it follows that
\begin{eqnarray*}
\sigma_A(H^k[G]) &=& \{\gamma^{[n^kg_2]}_2(G),\dots, \gamma^{[n^kg_s]}_s(G)\} \cup \{r_k\} \cup \Lambda_k,\\
\text{where }\;\;\Lambda_k  &=& \bigcup_{i=0}^{k-1}{\left\{(mn^{i}\gamma_2(H)+ r_i)^{[n^{k-1-i}h_2]}, \dots, (mn^i\gamma_t(H)+ r_i)^{[n^{k-1-i}h_t]}\right\}}.
\end{eqnarray*}
\end{proof}

\begin{example}\label{ex_2}
 For the graphs of Figure~\ref{graphG1}, we have $m=2$, $n=4$, $p=1$ and $q=2$.
From  Theorem~\ref{main_th_1}, we obtain the following degree,
order and spectra for $C_4^k[K_2]$ for a given $k \geq 1$ integer:
\begin{eqnarray*}
r_k                    &=& 2 \times 2 \frac{4^k-1}{4-1} + 1 = 
  \frac{4^{k+1}-1}{3},\\
\nu_k=\nu(H^k[G]) &=& 2 \times 4^k,\\
\sigma_A(H^k[G])  &=& \left\{(-1)^{[4^k]}\right\} \cup \left\{\frac{4^{k+1}-1}{3}\right\} \cup \Lambda_k,
\end{eqnarray*}
where $\Lambda_k = \bigcup_{i=0}^{k-1}{\left\{\left(2 \times 4^{i} \times 0+ \frac{4^{i+1}-1}{3}\right)^{[4^{k-1-i}\times 2]}, \left(2\times 4^i(-2)+ \frac{4^{i+1}-1}{3}\right)^{[4^{k-1-i}]}\right\}}$ $=$
      $\bigcup_{i=0}^{k-1}{\left\{(\frac{4^{i+1}-1}{3})^{[4^{k-1-i} \times 2]}, (-4^{i+1} + \frac{4^{i+1}-1}{3})^{[4^{k-1-i}]}\right\}}.$

In particular, for $k=2$ (the graph of Figure~\ref{graphG2}), it follows that
\begin{eqnarray*}
r_2                &=& 
21,\\
\nu_2=\nu(H^2[G])  &=& 32 \  \  \mbox{and}  \\
\sigma_A(H^2[G]) &=& \{21, (5)^{[2]}, (1)^{[8]}, (-1)^{[16]}, (-3)^{[4]},-11\}.
\end{eqnarray*}
\end{example}
\vskip 0.3cm

We may consider the graph obtained by an arbitrary number of iterations of the lexicographic product of a graph
for itself. In fact, for a given $q$-regular graph $H$ of order $n$, we assume that $H^0=K_1$, $H^1=H$ and that
$H^k=H^{k-1}[H] $ for $k\geq2$. Then, as immediate consequence of Theorem~\ref{main_th_1}, we have the following
corollary.

\begin{corollary}\label{cor_1_main_th_1}
Let $H$ be a connected $q$-regular graph of order $n$ such that $\sigma_A(H)=\{q,\gamma^{[h_2]}_2(H),\dots,\gamma^{[h_t]}_t(H)\}$.
Then, for each integer $k \geq 1$, $H^k$ is a $r_k$-regular graph of order $\nu_k$, such that
\begin{eqnarray*}
r_k           &=& q \frac{n^k-1}{n-1},\\
\nu_k         &=& n^k  \  \  \mbox{and}  \\
\sigma_A(H^k) &=& \left(\bigcup_{i=0}^{k-1}{\{(n^{i}\gamma_2(H)+r_i)^{[n^{k-1-i}h_2]},\dots,
                                         (n^i\gamma_t(H)+r_i)^{[n^{k-1-i}h_t]}\}}\right) \cup \{r_k\}.
\end{eqnarray*}
\end{corollary}

\begin{remark}\label{least_eigenvalue_power_k}
The least eigenvalue of $H^k$ is $\lambda_{n^k}(H^k)=n^{k-1}\lambda_n(H) + q\frac{n^{k-1}-1}{n-1}$.
\end{remark}
\begin{proof}
In fact, based on the Corollary~\ref{cor_1_main_th_1}, we obtain
\begin{eqnarray*}
\lambda_{n^k}(H^k) & = & \min_{0 \le i \le k-1} \{n^i\gamma_t(H) + r_i\} =\min_{0 \le i \le k-1} \{n^i\gamma_t(H) + q\frac{n^i-1}{n-1}\} \nonumber \\
                   & = & n^{k-1}\gamma_t(H) + q\frac{n^{k-1}-1}{n-1}
                         =n^{k-1}\lambda_n(H) + q\frac{n^{k-1}-1}{n-1}.
\end{eqnarray*}
The third equality above is obtained taking into account that for every $i \in \{0, \dots, k-1\}$,
$n^{k-1}\gamma_t(H) + q\frac{n^{k-1}-1}{n-1} \le n^i\gamma_t(H) + q\frac{n^i-1}{n-1}$ $\Leftrightarrow$ $(n^{k-1}-n^i)\gamma_t(H) \le - q \frac{n^{k-1}-n^i}{n-1}$
$\Leftrightarrow$ $\gamma_t(H) \le - \frac{q}{n-1}$ and last inequality holds since the graph $H$ has at least one edge and then (see \cite{Cve1})
$\gamma_t(H) \le -1 \le - \frac{q}{n-1}$.
\end{proof}

\begin{remark}\label{the_invariant_spectrum}
Let $H$ be a $p$-regular graph of order $n$. Then for all $k \in \mathbb{N}$
and for all nonnegative integer $q$, $\sigma_A(H^k)\setminus \{r_k\} \subset \sigma_A(H^{k+q})$, where $r_k$
is the regularity of $H^k$ and this inclusion means that all eigenvalues with the respective multiplicities
of the multiset $\sigma_A(H^k)\setminus \{r_k\}$ belong to the multiset $\sigma_A(H^{k+q})$.
\end{remark}

\begin{proof}
This is a direct consequence of Corollary~\ref{cor_1_main_th_1}.
\end{proof}

\begin{example}
Let us apply the Corollary~\ref{cor_1_main_th_1} to the powers of the Pertersen graph $P^k$, with $k \in \{2, 3, 100\}$.

\begin{center}
\begin{tabular}{ |c|l| } \hline
$k$ & Spectrum of $P^k$ \\ \hline
$k=1$ & $3, \qquad  \qquad \qquad \;\;\; 1^{[5]}, \qquad \;  -2^{[4]}$ \\ \hline
$k=2$ & $33, \;\;\; \quad \qquad 13^{[5]}, \; 1^{[50]}, \;\;\;\;\;\;\: -2^{[40]},\;\;\; -17^{[4]}$ \\ \hline
$k=3$ & $333, \; 133^{[5]},\; 13^{[50]}, 1^{[500]}, \;\;\;\;\;\; -2^{[400]},\; -17^{[40]},\; -167^{[4]}$ \\ \hline
\multirow{5}{4em}{$k=100$} & $\displaystyle  3 \times \sum_{i=0}^{99} 10^{i}, \;\displaystyle \qquad \;\; 1^{[5 \times 10^{99}]}, \;\displaystyle -2^{[4 \times 10^{99}]}$, \\
                           & $\displaystyle \;\;\;\; \left(\;\;\;\;\;\; 10^m + 3 \sum_{i=0}^{m-1} 10^{i}\right)^{[5 \times 10^{99-m}]}, \; m=1, \ldots, 99$,  \\
                           & $\displaystyle - \left(7+ 10^m + 6 \sum_{i=1}^{m-1} 10^{i} \right)^{[4 \times 10^{99-m}]}, \; m=1, \ldots, 99$.   \\
\hline
\end{tabular}
\end{center}

Notice that the graph $P^k$ has $10^k$ vertices, in particular $P^{100}$ has the googol number of vertices $10^{100}$. All the computations were done by Mathematica and lasted
just a few seconds.
\end{example}

\subsection{The Laplacian spectra}\label{laplacian}

In this section we characterize the Laplacian spectrum of the iterated lexicographic product $H^k[G]$, where $G$ and $H$ are arbitrary graphs.
The particular cases of the Laplacian spectra of these iterated lexicographic products, when $H$ is regular and when $H$ is arbitrary but equal to $G$ are also
presented.

\begin{theorem}\label{main_th_2}
Let $G$ be a graph of order $m$ such that $\sigma_L(G)=\{\mu_1(G), \ldots,$ $ \mu_m(G)\}$ and let $H$ be a graph such that
$V(H)=[n]$ and $\sigma_L(H)=\{\mu_1(H),\ldots,$ $\mu_n(H)\}$. Then, for each integer $k \geq 1$,
$H^k[G]$ is a graph of order  $\nu_k = m n^k$ such that
$$
\sigma_L({H^k[G]}) = \Omega_G^k \cup \Gamma_H^k
$$
where
\begin{eqnarray*}
\Omega_G^k &=& \bigcup_{(j_1,j_2,\dots, j_k) \in [n]^k} \left\{ \mu_l(G) + m \sum_{i=1}^k n^{i-1} d_H(j_i) : 1 \leq l \leq m-1 \right\}   \mbox{           and}\\
\Gamma_H^k &=& \bigcup_{i=2}^{k} \left(\bigcup_{(j_{i},\ldots, j_k) \in [n]^{k-i+1}}\left\{mn^{i-2}\mu_l(H) +
m \sum_{r=i}^k n^{r-1} d_H(j_r) : 1 \leq l \leq n-1 \right\}\right) \cup \\
             & & \left\{  m n^{k-1} \mu_j(H) : 1 \leq j \leq n \right\}.
\end{eqnarray*}
\end{theorem}

\begin{proof}
Corollary~\ref{corollary_2} give us the assertion in case $k=1$.
Given an integer $k \geq 2$, let suppose that the $$
\sigma_L({H^{k-1}[G]}) = \Omega_G^{k-1} \cup \Gamma_H^{k-1},
$$
where
\begin{eqnarray*}
\Omega_G^{k-1} &=& \bigcup_{(j_1,\ldots, j_{k-1}) \in [n]^{k-1}} \left\{ \mu_l(G) + m \sum_{i=1}^{k-1} n^{i-1} d_H(j_i) : 1 \leq l \leq m-1 \right\}   \mbox{           and}\\
\Gamma_H^{k-1} &=& \bigcup_{i=2}^{k-1} \left(\bigcup_{(j_{i},\ldots, j_{k-1}) \in [n]^{k-i}}\left\{mn^{i-2}\mu_l(H) +
m \sum_{r=i}^{k-1} n^{r-1} d_H(j_r) : 1 \leq l \leq n-1 \right\}\right)\\
 \cup & & \left\{  m n^{k-2} \mu_j(H) : 1 \leq j \leq n \right\}.
\end{eqnarray*}
Then, by Corollary~\ref{corollary_2},
\begin{eqnarray*}
\sigma_L(H^k[G]) &=& \sigma_L(H[H^{k-1}[G]]) =
\end{eqnarray*}
 $$ =  \left(\bigcup_{j_{k}=1}^n \left\{ mn^{k-1}d_{H}(j_k) + x : x \in \Omega_G^{k-1} \right\}\right) \cup
                \left(\bigcup_{j_{k}=1}^n \left\{ mn^{k-1}d_{H}(j_k) + y : y \in \Gamma_H^{k-1} \right\}\right),$$
 where
 $$
 \bigcup_{j_{k}=1}^n \left\{ mn^{k-1}d_{H}(j_k) + x : x \in \Omega_G^{k-1} \right\}= $$
  $$ =   \bigcup_{j_{k}=1}^{n} {\left(\bigcup_{(j_1,\ldots, j_{k-1}) \in [n]^{k-1}} {\left\{
     mn^{k-1}d_{H}(j_k) + \mu_l(G) +m \sum_{i=1}^{k-1} n^{i-1} d_H(j_i)  :  1 \leq l \leq n\right\}}\right)}$$
  $$ = \bigcup_{(j_1,\ldots, j_{k-1}, j_k) \in [n]^{k}} {\left\{
      \mu_l(G) +m \sum_{i=1}^{k} n^{i-1} d_H(j_i)  :  1 \leq l \leq n\right\}}=\Omega_G^{k}  \  \  \mbox{ and       }
  $$
  $$  \bigcup_{j_{k}=1}^n \left\{ mn^{k-1}d_{H}(j_k) + y : y \in \Gamma_H^{k-1} \right\} =
  $$
  $$= \bigcup_{j_{k}=1}^n \left(\bigcup_{i=2}^{k-1} {\bigcup_{(j_{i},\ldots, j_{k-1}) \in [n]^{k-i}} {\left\{mn^{k-1}d_{H}(j_k) +mn^{i-2}\mu_l(H) + m \sum_{r=i}^{k-1} n^{r-1} d_H(j_r) :
   1 \leq l \leq n-1 \right\}}}\right.
    $$
    $$\cup  \left. \left\{m n^{k-1}d_{H}(j_k) +mn^{k-2}\mu_l(H) : 1 \leq l \leq n-1\right\} \right.\Bigg)\cup \{mn^{k-1}\mu_j(H) : 1 \leq j \leq n\}= $$
$$ =\bigcup_{i=2}^{k} {\bigcup_{(j_{i},\ldots, j_{k}) \in [n]^{k-i+1}} {\left\{mn^{i-2}\mu_l(H) + m \sum_{r=i}^{k} n^{r-1} d_H(j_r) : 1 \leq l \leq n-1 \right\}}}\cup
    $$
    $$  \{mn^{k-1}\mu_j(H) ; 1 \leq j \leq n\}= \Gamma_H^{k}\,.
$$
\end{proof}

As immediate consequence of the above theorem, for a  regular graph $H$  it follows

\begin{corollary}\label{cor_2_main_th_2_regular}
Let $G$ and $H$ as in Theorem~\ref{main_th_2}, with $H$ is $q$-regular. Then, for each integer $k \geq 1$,
 \begin{eqnarray*}
\sigma_L({H^k[G]}) &=& \left\{ \left(\mu_l(G) + mq \frac{n^{k} - 1}{n-1} \right)^{[n^k]}: 1 \leq l \leq m-1 \right\} \cup \{0\} \cup\\
& &  \    \bigcup_{i=2}^{k+1} {\left\{\left(mn^{i-2}\mu_l(H) +
mq  n^{i-1}\frac{n^{k-i+1}-1}{n-1}\right)^{[n^{k-i+1}]} : 1 \leq l \leq n-1 \right\}}.
\end{eqnarray*}
\end{corollary}

\begin{proof}
From Theorem~\ref{main_th_2}, for all integer $k \geq 1$, it follows that
$$
\sigma_L({H^k[G]}) = \Omega_G^k \cup \Gamma_H^k,
$$
where
\begin{eqnarray*}
\Omega_G^k &=& \bigcup_{(j_1,\ldots, j_k) \in [n]^{k}} \left\{ \mu_l(G) + mq \sum_{i=1}^k n^{i-1}  : 1 \leq l \leq m-1 \right\}   \\
&=&  \left\{ \left(\mu_l(G) + mq \frac{n^{k} - 1}{n-1} \right)^{[n^k]}: 1 \leq l \leq m-1 \right\}    \   \   \   \   \mbox{   and    }\\
\end{eqnarray*}
 \begin{eqnarray*}
\Gamma_H^k &=& \bigcup_{i=2}^{k} \left(\bigcup_{(j_{i},\ldots, j_k) \in [n]^{k-i+1}}\left\{mn^{i-2}\mu_l(H) +
mq \sum_{r=i}^k n^{r-1} : 1 \leq l \leq n-1 \right\}\right) \cup \\
             & & \left\{  m n^{k-1} \mu_j(H) : 1 \leq j \leq n \right\} \\
             &=& \bigcup_{i=2}^{k+1} {\left\{\left(mn^{i-2}\mu_l(H) +
mq  n^{i-1}\frac{n^{k-i+1}-1}{n-1}\right)^{[n^{k-i+1}]} : 1 \leq l \leq n-1 \right\}}\cup \{0\}.
\end{eqnarray*}
\end{proof}

Now, let us consider the case $G=H$.

\begin{corollary}\label{cor_3_main_th_2}
Let $H$ be a graph such that $V(H)=[n]$, with $\sigma_L(H)=\{\mu_1(H),$ $\ldots,$ $\mu_n(H)\}$. Then $H^k$ is a graph of order $\nu_k=n^k$  such that
\begin{eqnarray*}
\sigma_L(H^k) &=& \bigcup_{i=1}^{k-1} \left(\bigcup_{(j_{i},\ldots, j_{k-1}) \in [n]^{k-i}}\left\{n^{i-1}\mu_l(H) +
\sum_{r=i}^{k-1} n^{r} d_H(j_r) : 1 \leq l \leq n-1 \right\}\right) \cup \\
             & & \left\{  n^{k-1} \mu_j(H) : 1 \leq j \leq n \right\},
\end{eqnarray*}
for all $k \ge 2$.
\end{corollary}

\begin{proof}
The first statement is obvious. Regarding the second statement, applying again Theorem~\ref{main_th_2} for $k \geq 2$ we obtain
$$\sigma_L(H^k)  =  \sigma_L(H^{k-1}[H])=$$
              $$=  \bigcup_{(j_1,j_2,\ldots,j_{k-1}) \in [n]^{k-1}}\left\{{\mu}_l(H) +n\sum_{i=1}^{k-1} n^{i-1}d_H(j_i) : 1 \leq l \leq n-1 \right\}\cup $$
              $$ \bigcup_{i=2}^{k-1} \left( \bigcup_{(j_i,\ldots,j_{k-1})\in [n]^{k-i}}\left\{ n^{i-1}{\mu_l}(H)+ n\sum_{r=i}^{k-1}  n^{r-1} d_H(j_r) : 1 \leq l \leq n-1\right\}\right) \cup $$
               $$ \{n n^{k-2}{\mu_j}(H) : 1 \leq j \leq n \}=$$
                $$=\bigcup_{i=1}^{k-1} \left( \bigcup_{(j_i,\ldots,j_k) \in [n]^{k-i}}\left\{n^{i-1}{\mu_l}(H) +\sum_{r=i}^{k-1}n^{r} d_H(j_r) : 1 \leq l \leq n-1 \right\} \right) \cup $$
                $$\{ n^{k-1}{\mu_j}(H): 1 \leq j \leq n\}.
$$
\end{proof}

Finally, the next proposition determines the algebraic connectivity and the largest Laplacian eigenvalue of $H^k$, for $k \ge 1$.

\begin{proposition}\label{algebraic_connectivity}
If $H$ is a connected graph of order $n$ with $\sigma_L(H)=\{{\mu_1}(H),\ldots,$ ${\mu_n}(H)\}$ and $k \ge 1$, then
\begin{eqnarray}
\mu_{n^{k}-1}(H^k) &=& n^{k-1}\mu_{n-1}(H)  \  \  \  \   \mbox{ and }\label{algebraic_connectivity}\\
\mu_1(H^k)         &=& n^{k-1}\mu_{1}(H)  \label{laplacian_index}
\end{eqnarray}
\end{proposition}

\begin{proof}
Let $k \geq 1$ be fixed. From Corollary~\ref{cor_3_main_th_2}, it follows that the second least eigenvalue of $H^k$
is among the values $n^{k-1}\mu_{n-1}(H)$ and $n^{i-1}\mu_{n-1}(H) + \sum_{r=i}^{k-1} n^rd_H(j_r)$ for $1 \leq i \leq k-1$.
  We may recall that $\delta(H) \geq \mu_{n-1}(H)$; then,
 for all $1 \leq i \leq k-1$, it holds that
$$
n^{i-1}\mu_{n-1}(H) + \sum_{r=i}^{k-1} n^rd_H(j_r) \geq n^{i-1}\mu_{n-1}(H) + \sum_{r=i}^{k-1} n^r \delta(H) $$ $$ \geq  n^{i-1}\mu_{n-1}(H) + \mu_{n-1}(H)\sum_{r=i}^{k-1} n^r
        =\mu_{n-1}(H)\left( n^{i-1} +  \sum_{r=i}^{k-1} n^r\right)
$$
$$
= \ \mu_{n-1}(H)\sum_{r=i-1}^{k-1} n^r =   \mu_{n-1}(H)\sum_{r=i }^{k } n^{r-1} \  \geq  \ \mu_{n-1}(H)n^{k-1}.
$$
Thus the equality \eqref{algebraic_connectivity} is proved. Now, let us prove the equality \eqref{laplacian_index}.\\
Applying again Corollary~\ref{cor_3_main_th_2}, it follows that the largest Laplacian eigenvalue of $H^k$ is among the values $n^{k-1}\mu_{1}(H)$ and
$n^{i-1}\mu_{1}(H)+\sum_{r=i}^{k-1}{n^r d_H(j_r)}$, $1 \leq i \leq k-1$. Since
$\mu_1(H) \ge \Delta(H)+1$, for $1 \le i \le k-1$, it follows that
\begin{eqnarray*}
n^{i-1}\mu_{1}(H)+\sum_{r=i}^{k-1}{n^r d_H(j_r)} & \le & n^{i-1}\mu_{1}(H)+\sum_{r=i}^{k-1}{n^r(\mu_1(H)-1)}\\
                                                 &  =  & \mu_{1}(H)\sum_{r=i-1}^{k-1}{n^r}-\sum_{r=i}^{k-1}{n^r}\\
                                                 &  =  & \mu_{1}(H)n^{i-1}\frac{n^{k-i+1}-1}{n-1} - n^i \frac{n^{k-i}-1}{n-1}\\
                                                 &  =  & \mu_{1}(H)n^{i-1}(\frac{n^{k-i}-1}{n-1} + n^{k-i}) - n^i \frac{n^{k-i}-1}{n-1}\\
                                                 &  =  & \mu_{1}(H)n^{k-1} + \frac{n^{k-i}-1}{n-1}n^{i-1}(\mu_{1}(H)-n)\\
                                                 & \le & n^{k-1}\mu_{1}(H)
\end{eqnarray*}
The last inequality is obtained taking into account that $\mu_{1}(H)-n \le 0$.
\end{proof}

\section{Spectral and combinatorial invariant properties of lexicographic powers of graphs}\label{extensions}

In this section, a few well known spectral and combinatorial invariant properties of a graph $H$ are extended to the lexicographic powers of $H$.
For instance, considering that $H$ has order $n \ge 2$, for all $k \geq 1$, we may deduce that
\begin{equation}
\delta(H^k) = \delta(H) \frac{n^{k}-1}{n-1} \qquad \left(\Delta(H^k) = \Delta(H) \frac{n^{k}-1}{n-1}\right).\label{minimum_maximum_degrees}
\end{equation}
Notice that since $H$ has order $n$, then $H^k$ has order $n^k$. The equalities \eqref{minimum_maximum_degrees} can be proved by induction on $k$, taking into
account that they are obviously true for $k=1$. Assuming that the equalities \eqref{minimum_maximum_degrees} are true for $k-1$, with $k \ge 2$, it is immediate
that a vertex of $H^k$ with minimum (maximum) degree is a minimum (maximum) degree vertex of the copy of $H^{k-1}$ located in the position of a minimum (maximum)
degree vertex of $H$, and then its  degree in $H^k$ is equal to $\delta(H) (\frac{n^{k-1}-1}{n-1}+n^{k-1})$ \Big($\Delta(H)( \frac{n^{k-1}-1}{n-1}+n^{k-1})$\Big).\\

For an arbitrary graph $G$, let $q_1(G)$ and $q_n(G)$ be the largest and the least eigenvalue of the signless Laplacian matrix of $G$ (that is, the matrix $A_G+D$),
respectively. Taking into account the relations  $2\delta(G) \le  q_1(G) \le 2\Delta(G)$, which were proved in \cite{CveRowSimic2007}, and also the inequality
$q_n(G) < \delta(G)$ \cite{Das2010}, for the lexicographic power $k$ of a graph $H$ we obtain  the inequalities
\begin{eqnarray*}
2 \delta(H) \frac{n^k-1}{n-1} \; \le \: q_1(H^k) & \le & 2 \Delta(H)\frac{n^k-1}{n-1}  \  \  \   \  \   \   \mbox{and} \\
q_n(H^k) & < & \delta(H) \frac{n^k-1}{n-1}.
\end{eqnarray*}

Denoting the distance between two vertices $x$ and $y$ in $G$ by $d_G(x,y)$ and the diameter of $G$ by $\text{diam}(G)$, we may conclude the following interesting result concerning the diameter of the iterated lexicographic products of graphs.

\begin{proposition}
Let $H$ be a connected not complete graph and let $G$ be an arbitrary graph of order $m$. For very $k \in \mathbb{N}$
$$
diam(H^{k+1}) = diam(H^k[G]) = \text{diam}(H).
$$
\end{proposition}

\begin{proof}
Consider $V(H)=\{1, \dots, n\}$ and $x, y \in V(H^k[G])$ ($x, y \in V(H^{k+1})$). Then we have two cases (a) they are both in
the same copy of $H^{k-1}[G]$ ($H^{k}$) located in the position of the vertex $i \in V(H)$ or (b) they are in
different copies of $H^{k-1}[G]$ ($H^{k}$) located in the positions of the vertices $r, s \in V(H)$.
\begin{enumerate}
\item[(a)] If $x$ and $y$ are adjacent, then $d_{H^k[G]}(x,y)=1$ ($d_{H^{k+1}}(x,y)=1$), otherwise since there exists a vertex
           $j \in V(H)$ such that $ij \in E(H)$ and then there is a path $x,z,y$, where $z$ is a vertex of the copy of $H^{k-1}[G]$
           ($H^k$) located in the position of the vertex $j \in V(H)$. Therefore, $d_{H^k[G]}(x,y)=2$ ($d_{H^{k+1}}(x,y)=2$).
\item[(b)] In this case, assuming that $r, j_1, \dots, j_t, s$ is a shortest path in $H$ connecting the vertices $r$ and $s$, there are vertices
           $z_1, \dots, z_t$ in the copies of $H^{k-1}[G]$ ($H^k$) located in the positions of the vertices $j_1, \dots, j_t$, respectively,
           such that $x, z_1, \dots, z_t, y$ is a path of length $d_H(r,s)$.
\end{enumerate}
\end{proof}

\subsection{The stability number}
Regarding the stability number $\alpha(G)$ (the maximum cardinality of a vertex subset of an arbitrary graph $G$ with pairwise nonadjacent vertices),
according to \cite{GellerStahl75}, $\alpha(H[G]) = \alpha(H)\alpha(G)$ for an arbitrary graph $H$. Thus we may conclude that $\alpha(H^k)=\alpha(H)^k$
(and, denoting the complement of graph $F$ by $\overline{F}$ and the clique number by $\omega(F)$,
since $\overline{H[G]}= \overline{H}[\overline{G}]$, $\omega(H^k) = \omega(H)^k$). Furthermore, from the spectral upper bound
$\alpha(G) \le n \frac{\mu_1(G) - \delta(G)}{\Delta(G)},$ independently deduced in \cite{LLT2007} and \cite{GN08} for an arbitrary graph $G$, and taking
into account \eqref{minimum_maximum_degrees} and \eqref{laplacian_index},  considering the $k$-th lexicographic power of a graph $H$ of order $n$ we obtain
\begin{equation*}
\alpha(H^k) \le n^k \frac{\mu_1(H^k) - \delta(H^k)}{\Delta(H^k)} \le n^k \frac{\frac{n-1}{n^k-1}n^{k-1}\mu_1(H)- \delta(H)}{\Delta(H)}.
\end{equation*}

\subsection{The vertex connectivity}
Considering a graph $G$ of order $m$ and a graph $H$ of order $n$, it is well known that the lexicographic product $H[G]$ is connected if and only if $H$
is a connected graph \cite{harary_wilcox_67}. On the other hand, according to \cite{GellerStahl75}, if both $G$ and $H$ are not complete, then
$\upsilon(H[G])=m \upsilon(H)$, where $\upsilon(H)$ denotes the vertex connectivity of $H$ (that is, the minimum number of vertices whose
removal yields a disconnected graph). Therefore, $\upsilon(H^k) = n^{k-1}\upsilon(H).$ Furthermore, we may conclude that when $H$ is connected not complete
(and then $H^k$ is also connected not complete),
$$
n^{k-1}\mu_{n-1}(H) \le \upsilon(H^k) \le \delta(H)\frac{n^{k}-1}{n-1}.
$$
In fact, it should be noted that $\upsilon(G) \le \delta(G)$ and, when $G$ is not complete, $\mu_{n-1}(G) \le \upsilon(G)$, see \cite{Fiedler73}.
Therefore, taking into account \eqref{algebraic_connectivity} and \eqref{minimum_maximum_degrees} we obtain
$n^{k-1}\mu_{n-1}(H) = \mu_{n-1}(H^k) \le \upsilon(H^k) \le \delta(H^k) = \delta(H)\frac{n^k-1}{n-1}$.

\subsection{The chromatic number}
Concerning the relations of the chromatic number of a graph $G$ of order $n$ with its spectrum, the following
lower bound due to Hoffman in \cite{hoffman_79} is well known.
\begin{eqnarray*}
\chi(G) & \ge & 1 - \frac{\lambda_1(G)}{\lambda_n(G)}. \label{hoffman}
\end{eqnarray*}
As direct consequence, if a graph $H$ is $q$-regular of order $n$, taking into account the Remark~\eqref{least_eigenvalue_power_k},
we may conclude the following lower bound on the chromatic number of $H^k$:
\begin{eqnarray*}
\chi(H^k) \ge 1 - \frac{r_k}{\lambda_{n^k}(H^{k})} &=& 1 - q\frac{n^k-1}{(n-1)\left(n^{k-1}\lambda_n(H)+q\frac{n^{k-1}-1}{n-1}\right)}\\
                                                   &=& 1 - \frac{n^k-1}{n^{k-1}\left((n-1)\frac{\lambda_n(H)}{q}+1\right) - 1}.
\end{eqnarray*}

\section*{Acknowledgement}
The research of Nair Abreu is partially  supported by Project Universal CNPq 442241/2014 and Bolsa PQ 1A CNPq, 304177/2013-0.
The research of Domingos M. Cardoso and Paula Carvalho is supported by the Portuguese Foundation for Science
and Technology (\textquotedblleft FCT-Funda\c c\~ao para a Ci\^encia e a Tecnologia\textquotedblright),
through the CIDMA - Center for Research and Development in Mathematics and Applications, within project
UID/MAT/04106/2013. Cybele Vinagre thanks the support of FAPERJ,  through APQ5 210.373/2015 and the hospitality of
Department of Mathematics of University of Aveiro, Portugal, where this paper was finished.

\end{document}